\pdfoutput=1
\RequirePackage{ifpdf}
\ifpdf 
\documentclass[pdftex]{sigma}
\else
\documentclass{sigma}
\fi

\numberwithin{equation}{section}

\newtheorem{Theorem}{Theorem}[section]
\newtheorem{Corollary}[Theorem]{Corollary}
\newtheorem{Lemma}[Theorem]{Lemma}
\newtheorem{Proposition}[Theorem]{Proposition}
 { \theoremstyle{definition}
\newtheorem{Example}[Theorem]{Example} }

\begin{document}

\newcommand{\arXivNumber}{1801.00506}

\renewcommand{\thefootnote}{}

\renewcommand{\PaperNumber}{047}

\FirstPageHeading

\ShortArticleName{On the Strong Ratio Limit Property for Discrete-Time Birth-Death Processes}

\ArticleName{On the Strong Ratio Limit Property\\ for Discrete-Time Birth-Death Processes\footnote{This paper is a~contribution to the Special Issue on Orthogonal Polynomials, Special Functions and Applications (OPSFA14). The full collection is available at \href{https://www.emis.de/journals/SIGMA/OPSFA2017.html}{https://www.emis.de/journals/SIGMA/OPSFA2017.html}}}

\Author{Erik A.~VAN DOORN}

\AuthorNameForHeading{E.A.~van Doorn}

\Address{Department of Applied Mathematics, University of Twente,\\
P.O. Box 217, 7500 AE Enschede, The Netherlands}
\Email{\href{mailto:e.a.vandoorn@utwente.nl}{e.a.vandoorn@utwente.nl}}
\URLaddress{\url{http://wwwhome.math.utwente.nl/~doornea/}}

\ArticleDates{Received January 03, 2018, in final form May 13, 2018; Published online May 15, 2018}

\Abstract{A sufficient condition is obtained for a discrete-time birth-death process to possess the {\em strong ratio limit property}, directly in terms of the one-step transition probabilities of the process. The condition encompasses all previously known sufficient conditions.}

\Keywords{(a)periodicity; birth-death process; orthogonal polynomials; random walk
measure; ratio limit; transition probability}

\Classification{60J80; 42C05}

\newcommand{\bQ}{\bar{Q}}
\newcommand{\tQ}{\tilde{Q}}

\renewcommand{\thefootnote}{\arabic{footnote}}
\setcounter{footnote}{0}

\section{Introduction}

In what follows $\mathcal{X} \equiv \{X(n),\, n=0,1,\ldots\}$ is a discrete-time birth-death process on $\mathcal{N}\equiv \{0,1,\dots\}$, with tridiagonal matrix of one-step transition probabilities
\begin{gather*}
P :=
\begin{pmatrix}
r_0 & p_0 & 0 & 0 & 0 & \ldots\cr
 q_1 & r_1 & p_1 & 0 & 0 & \ldots\cr
 0 & q_2 & r_2 & p_2 & 0 & \ldots\cr
 \ldots & \ldots & \ldots & \ldots & \ldots & \ldots\cr
 \ldots & \ldots & \ldots & \ldots & \ldots & \ldots\cr
\end{pmatrix}.
\end{gather*}
Following Karlin and McGregor \cite{KM59} we will refer to $\mathcal{X}$ as a~{\it random walk}. We assume throughout that $p_j>0$, $q_{j+1}>0$, $r_j\geq 0$, and $p_j+q_j+r_j = 1$ for $j\geq 0$, where $q_0 := 0$. We let
\begin{gather}\label{pi}
\pi_0 := 1,\qquad \pi_n := \frac{p_0p_1\cdots p_{n-1}}{q_1q_2\cdots q_n}, \qquad n\geq 1,
\end{gather}
and define the polynomials $Q_n$ via the recurrence relation
\begin{gather}
xQ_n(x)=q_nQ_{n-1}(x)+r_nQ_n(x)+p_nQ_{n+1}(x),\qquad n > 1,\nonumber\\
Q_0(x)=1,\qquad p_0Q_1(x)=x-r_0.\label{recQ}
\end{gather}

Karlin and McGregor \cite{KM59} have shown that the $n$-step transition probabilities
\begin{gather*}
P_{ij}(n):=\Pr\{X(n)=j\,|\,X(0)=i\} = (P^n)_{ij}, \qquad n\geq 0, \qquad i,j\in\mathcal{N},
\end{gather*}
may be represented in the form
\begin{gather}\label{rep}
P_{ij}(n) = \pi_j\int_{[-1,1]}x^nQ_i(x)Q_j(x)\psi(dx),
\end{gather}
where $\psi$ is the (unique) Borel measure on the interval $[-1,1]$, of total mass~1 and with infinite support, with respect to which the polynomials~$Q_n$ are orthogonal. Adopting the terminology of~\cite{DS93} we will refer to the measure $\psi$ as a {\it random walk measure}. Of particular interest to us is $\eta := \sup \operatorname{supp}(\psi)$, the largest point in the support of the random walk measure $\psi$, which may also be characterized in terms of the polynomials~$Q_n$ by
\begin{gather}
\label{pos}
x \geq \eta \ \Longleftrightarrow \ Q_n(x) > 0 \qquad \mbox{for all} \ \ n \geq 0
\end{gather}
(see, for example, Chihara \cite[Theorem II.4.1]{C78}). We will see in the next section that $\eta > 0$.

The random walk $\mathcal{X}$ is said to have the {\em strong ratio limit property $($SRLP$)$} if the limits
\begin{gather}\label{limits}
\lim_{n\to\infty} \frac{P_{ij}(n)}{P_{kl}(n)}, \qquad i,j,k,l \in \mathcal{N},
\end{gather}
exist simultaneously. The SRLP was introduced in the more general setting of discrete-time Markov chains on a countable state space by Orey~\cite{O61} and Pruitt~\cite{P65}, but the problem of finding conditions for the limits~\eqref{limits} to exist in the specific setting of random walks had been considered before in~\cite{KM59}. A satisfactory and comprehensive solution to the problem of finding conditions for the SRLP is still lacking,
even in the relatively simple setting at hand. So it remains a~challenge to find necessary and/or sufficient conditions. For more information on the history of the problem we refer to~\cite{DS95b} and~\cite{K95}.

In \cite[Theorem~3.1]{DS95b} a necessary and sufficient condition for the random walk $\mathcal{X}$ to have the SRLP has been given in terms of the associated random walk measure~$\psi$. Namely, letting
\begin{gather}\label{C}
C_n(\psi) := \frac{\int_{[-1,0)}(-x)^n\psi(dx)}{\int_{(0,1]} x^n\psi(dx)}, \qquad n \geq 0,
\end{gather}
the limits \eqref{limits} exist simultaneously if and only if
\begin{gather}\label{limC}
\lim_{n\to\infty}C_n(\psi) = 0,
\end{gather}
in which case we have
\begin{gather*}
\lim_{n\to\infty} \frac{P_{ij}(n)}{P_{kl}(n)} = \frac{\pi_jQ_i(\eta)Q_j(\eta)}{\pi_lQ_k(\eta)Q_l(\eta)}, \qquad i,j,k,l \in\mathcal{N}.
\end{gather*}
Note that the denominator in \eqref{C} is positive since $\eta>0$, so that $C_n(\psi)$ exists and is nonnegative for all $n$. Some sufficient conditions for \eqref{limC} -- and, hence, for $\mathcal{X}$ to possess the SRLP -- are also given in~\cite{DS95b}. In particular, \cite[Theorem~3.2]{DS95b} tells us that
\begin{gather}\label{QC}
\lim_{n\to\infty} |Q_n(-\eta)/Q_n(\eta)| = \infty \ \Rightarrow \ \lim_{n\to\infty} C_n(\psi) = 0.
\end{gather}
The reverse implication is conjectured in \cite{DS95b} to be valid as well.

In this paper we will prove a sufficient condition for $\mathcal{X}$ to have the SRLP directly in terms of the one-step transition probabilities. Concretely, we will establish the following result.

\begin{Proposition}\label{proposition}
If the random walk $\mathcal{X}$ satisfies
\begin{gather}\label{condition}
\sum_{j\geq 0} \frac{1}{p_j\pi_j} \sum_{k=0}^j r_k\pi_k =\infty,
\end{gather}
then $\lim\limits_{n\to\infty} |Q_n(-\eta)/Q_n(\eta)| = \infty$.
\end{Proposition}

Together with \eqref{QC} this result immediately leads to the following.

\begin{Theorem}\label{theorem}
If the random walk $\mathcal{X}$ satisfies~\eqref{condition} then $\mathcal{X}$ possesses the SRLP.
\end{Theorem}

We will see that Theorem~\ref{theorem} encompasses all previously obtained sufficient conditions for the SRLP.

The proof of Proposition~\ref{proposition} will be based on three lemmas. Lemma~\ref{bQlim} and a number of preliminary results related to the polynomials~$Q_n$ and the orthogonalizing measure $\psi$ are collected in the next section. Two further auxiliary lemmas are established in Section~\ref{aux}. The actual proof of Proposition~\ref{proposition} and some concluding remarks can be found in Section~\ref{conc}, which also contains an example showing that~\eqref{condition} is not {\em necessary} for the SRLP.

\section{Preliminaries}\label{prel}

Whitehurst \cite[Theorem~1.6]{W82} has shown that the random walk measure $\psi$ satisfies
\begin{gather}\label{rw}
\int_{[-1,1]} xQ_n^2(x)\psi(dx) \geq 0, \qquad n \geq 0,
\end{gather}
and, conversely, that any Borel measure $\psi$ on the interval $[-1,1]$, of total mass 1 and with infinite support, is a random walk measure if it satisfies~\eqref{rw} (see also \cite[Theorem~1.2]{DS93}). Evidently, \eqref{rw} implies $\eta = \sup \operatorname{supp}(\psi) >0$, but it can actually be shown (see, for example, \cite[Corollary~2 to Theorem~IV.2.1]{C78}) that
\begin{gather*}
\eta > r_j, \qquad j \in\mathcal{N}.
\end{gather*}
By \cite[Lemma~2.3]{DS95a} we also have
\begin{gather*}
\inf_j \{r_j+r_{j+1}\}\leq \inf \operatorname{supp}(\psi) + \eta \leq \sup_j\{r_j+r_{j+1}\}, \qquad j \in\mathcal{N},
\end{gather*}
so that $\inf \operatorname{supp}(\psi)\geq -\eta$, and hence $\operatorname{supp}(\psi) \subset [-\eta,\eta]$.

The measure $\psi$ is symmetric about $0$ if (and only if) the random walk $\mathcal{X}$ is {\em periodic}, that is, if $r_j=0$ for all $j$ (see \cite[p.~69]{KM59}). In this case we also have
\begin{gather*}
(-1)^nQ_n(-x) = Q_n(x), \qquad n \geq 0,
\end{gather*}
and it follows from \eqref{rep} that $P_{ij}(n) =0$ if $n+i+j$ is odd. Hence the limits in~\eqref{limits} will not exist if $\mathcal{X}$ is periodic, which is also reflected by the fact that $C_n(\psi) = 1$ for all $n$ in this case.

$\mathcal{X}$ is called {\em aperiodic} if it is not periodic. From Whitehurst \cite[Theorem~5.2]{W78} we have the subtle result
\begin{gather*}
\mathcal{X} \mbox{~is~aperiodic} \ \Rightarrow \ \int_{[-\eta,\eta]}\frac{\psi(dx)}{\eta+x} < \infty,
\end{gather*}
so that $\psi(\{-\eta\}) = 0$ if $\mathcal{X}$ is aperiodic.

We continue with some useful observations from the recurrence relations \eqref{recQ}. The first one is the {\em Christoffel--Darboux} identity
\begin{gather*}
p_n\pi_n(Q_n(x)Q_{n+1}(y)-Q_n(y)Q_{n+1}(x))=(y-x)\sum_{j=0}^n\pi_jQ_j(x)Q_j(y)
\end{gather*}
(see, for example, \cite[Theorem I.4.5]{C78}). Hence, by \eqref{pos},
\begin{gather}\label{etaleq}
\eta \leq x < y \ \Rightarrow \ Q_n(x)Q_{n+1}(y) > Q_n(y)Q_{n+1}(x) > 0 \qquad \mbox{for all} \ \ n \geq 0.
\end{gather}
Since $p_j+q_j+r_j=1$ for all $j$ it follows readily from \eqref{recQ} that $Q_n(1)=1$ for all $n$, so~\eqref{etaleq} leads to
\begin{gather} \label{mon1}
\eta \leq x < 1 \ \Rightarrow \ 0 < Q_{n+1}(x) < Q_n(x) < Q_0(x)=1 \qquad \mbox{for all} \ \ n \geq 1.
\end{gather}

Next, writing $\bQ_n(x):= (-1)^nQ_n(x)$, we see from \eqref{recQ} that
\begin{gather*}
p_n\pi_n(\bQ_{n+1}(x)-\bQ_n(x)) =p_{n-1}\pi_{n-1}(\bQ_n(x)-\bQ_{n-1}(x))\\
\hphantom{p_n\pi_n(\bQ_{n+1}(x)-\bQ_n(x)) =}{} +(2r_n-1-x)\pi_n \bQ_n(x), \qquad n \geq 1,\\
p_0\pi_0(\bQ_1(x)-\bQ_0(x)) = (2r_0-1-x)\pi_0 \bQ_0(x),
\end{gather*}
from which we readily obtain
\begin{gather*}
\bQ_{n+1}(x) = 1 + \sum_{j=0}^n \frac{1}{p_j\pi_j} \sum_{k=0}^j(2r_k-1-x) \pi_k \bQ_k(x), \qquad n \geq 0,
\end{gather*}
and hence
\begin{gather}\label{bQn1}
\bQ_{n+1}(-1) = 1 + 2\sum_{j=0}^n \frac{1}{p_j\pi_j} \sum_{k=0}^j r_k\pi_k \bQ_k(-1),\qquad n \geq 0.
\end{gather}
This equation, observed already by Karlin and McGregor \cite[p.~76]{KM59}, leads to the first of our three lemmas.

\begin{Lemma}\label{bQlim}
The sequence $\{(-1)^nQ_n(-1)\}_n$ is increasing, and strictly increasing for $n$ sufficiently large, if $($and only if$)$ $\mathcal{X}$ is aperiodic. Moreover,
\begin{gather}\label{aa}
\sum_{j\geq 0} \frac{1}{p_j\pi_j} \sum_{k=0}^j r_k\pi_k =\infty \ \Longleftrightarrow \ \lim_{n\to\infty}(-1)^nQ_n(-1) = \infty.
\end{gather}
\end{Lemma}

\begin{proof}Since $\bQ_0(-1)=1$, while, by \eqref{bQn1},
\begin{gather}\label{bQn}
\bQ_{n+1}(-1) = \bQ_n(-1) + \frac{2}{p_n\pi_n} \sum_{k=0}^n r_k\pi_k \bQ_k(-1), \qquad n \geq 0,
\end{gather}
the first statement is obviously true. So we have $\bQ_n(-1)\geq 1$, which, in view of \eqref{bQn1} implies the necessity in the second statement. To prove the sufficiency we let
\begin{gather*}
\beta_j := \frac{2}{p_j\pi_j} \sum_{k=0}^j r_k\pi_k, \qquad j \geq 0,
\end{gather*}
and assume that $\sum_j \beta_j$ converges. By \eqref{bQn} we then have
\begin{gather*}
\bQ_{n+1}(-1) \leq \bQ_n(-1)(1 + \beta_n), \qquad n \geq 0,
\end{gather*}
since $\bQ_n(-1)$ is increasing in $n$. It follows that
\begin{gather*}
\bQ_{n+1}(-1) \leq \prod_{j=0}^n (1 + \beta_j), \qquad n \geq 0.
\end{gather*}
But since $\prod_j (1 + \beta_j)$ and $\sum_j \beta_j$ converge together, we must have $\lim\limits_{n\to\infty}\bQ_n(-1) < \infty$.
\end{proof}

The above lemma also plays a central role in \cite{D18}, where the conditions in~\eqref{aa} are shown to be equivalent to {\em asymptotic aperiodicity} of the random walk. For completeness' sake we have included the proof.

We recall from \cite{KM59} that $\mathcal{X}$ is {\em recurrent}, that is, the probability, for any state, of returning to that state is one, if and only if
\begin{gather} \label{L}
L := \sum_{j\geq 0} \frac{1}{p_j\pi_j} = \infty.
\end{gather}
$\mathcal{X}$ is called {\em transient} if it is not recurrent. It has been shown in \cite{KM59} that
\begin{gather*}
\int_{[-\eta,\eta]}\frac{\psi(dx)}{1-x} = L,
\end{gather*}
so we must have $\eta=1$ if $\mathcal{X}$ is recurrent. From Lemma~\ref{bQlim} we now obtain
\begin{gather}\label{recurrent}
\mathcal{X} \ \mbox{is~aperiodic~and~recurrent} \ \Rightarrow \ \lim_{n\to\infty}(-1)^nQ_n(-1) = \infty,
\end{gather}
a result noted earlier by Karlin and McGregor \cite[p.~76]{KM59}. Considering \eqref{QC} and the fact that $\eta=1$ if $\mathcal{X}$ is recurrent, the conclusion in \eqref{recurrent} implies the SRLP, so that we have regained
\cite[Theorem~2]{KM59}. (This result was later generalized to {\em symmetrizable} Markov chains by Orey \cite[Theorem~2]{O61}.) For later use we also note that
\begin{gather}\label{ineq}
\sum_{j\geq 0} \frac{1}{p_j\pi_j} \sum_{k=0}^j r_k\pi_k \geq \sum_{j\geq 0} \frac{r_j}{p_j},
\end{gather}
so that, by Lemma~\ref{bQlim},
\begin{gather} \label{rjpj}
\sum_{j\geq 0} \frac{r_j}{p_j} = \infty \ \Rightarrow \ \lim_{n\to\infty}(-1)^nQ_n(-1) = \infty.
\end{gather}

\section{Two auxiliary lemmas}\label{aux}

Throughout this section $\theta$ is a fixed number satisfying $\theta\geq \eta$. Defining $q_0(\theta):=0$ and
\begin{gather}\label{probaux}
p_j(\theta) := \frac{Q_{j+1}(\theta)}{Q_j(\theta)}\frac{p_j}{\theta},\qquad
r_j(\theta) := \frac{r_j}{\theta},\qquad
q_{j+1}(\theta) := \frac{Q_j(\theta)}{Q_{j+1}(\theta)}\frac{q_{j+1}}{\theta},
\qquad j \in \mathcal{N},
\end{gather}
the parameters $p_j(\theta)$, $q_j(\theta)$ and $r_j(\theta)$ satisfy $p_j(\theta)>0$, $q_{j+1}(\theta)>0$, $r_j(\theta)\geq 0$, and $p_j(\theta)+q_j(\theta)+r_j(\theta)=1$, so that they may be interpreted as the one-step transition probabilities of a~random walk $\mathcal{X}_\theta$ on $\mathcal{N}$. Denoting the corresponding polynomials by $Q_n(\cdot ;\theta)$ it follows readily that
\begin{gather}\label{newQ}
Q_n(x;\theta) = \frac{Q_n(\theta x)}{Q_n(\theta)}, \qquad n \geq 0,
\end{gather}
so that the associated measure $\psi_\theta$ satisfies
\begin{gather*}
\psi_\theta([-1,x]) = \psi([-\theta,x\theta]), \qquad -1 \leq x \leq 1.
\end{gather*}
Evidently, we have
\begin{gather*}
\eta(\theta) := \sup \operatorname{supp}(\psi_\theta) = \eta\theta^{-1}\leq 1,
\end{gather*}
while the analogues $\pi_n(\theta)$ of the constants $\pi_n$ of \eqref{pi} are easily seen to satisfy
\begin{gather}\label{newpi}
\pi_n(\theta) = \pi_n Q_n^2(\theta), \qquad n \geq 0.
\end{gather}
(In \cite[Appendix~2]{DS95a}) the special case $\theta=\eta$ is considered.) Obviously, $\mathcal{X}_\theta$ is periodic if and only if~$\mathcal{X}$ is periodic. Note that by choosing $\theta=1$ we return to the setting of the previous sections.

We have seen in Lemma~\ref{bQlim} that $(-1)^nQ_n(-1;\theta)$ is increasing, and strictly increasing for~$n$ sufficiently large, if $\mathcal{X}_\theta$ is aperiodic, or, equivalently, $\mathcal{X}$ is aperiodic. It thus follows from~\eqref{newQ} that $|Q_n(\theta)/Q_n(-\theta)|$ is decreasing, and strictly decreasing for $n$ sufficiently large, if $\mathcal{X}$ is aperiodic. Since $Q_n(-x;\theta) = (-1)^nQ_n(x;\theta)$ if~$\mathcal{X}_\theta$ is periodic, we conclude the following.

\begin{Lemma}\label{bounds}Let $\theta\geq\eta$. If $\mathcal{X}$ is periodic then $|Q_n(\theta)/Q_n(-\theta)| =1$ for all~$n$. If $\mathcal{X}$ is aperiodic then $|Q_n(\theta)/Q_n(-\theta)|$ is decreasing and tends to a~limit satisfying
\begin{gather*}
0 \leq \lim_{n\to\infty}|Q_n(\theta)/Q_n(-\theta)| < 1.
\end{gather*}
\end{Lemma}

In what follows we let
\begin{gather}\label{Mntheta}
M_n(\theta):=\sum_{j=0}^n \frac{1}{p_j(\theta)\pi_j(\theta)}\sum_{k=0}^j r_k(\theta)\pi_k(\theta), \qquad 0\leq n \leq \infty,
\end{gather}
so that in particular $M_\infty(1)$ equals the left-hand side of~\eqref{condition}. In combination with Lemma~\ref{bQlim}, interpreted in terms of $\mathcal{X}_\theta$, Lemma~\ref{bounds} gives us the next result.
\begin{Corollary}\label{corollary}
For $\theta\geq\eta$ we have
\begin{gather}\label{rho0eq}
M_\infty(\theta) =\infty \ \Longleftrightarrow \ \lim_{n\to\infty}|Q_n(\theta)/Q_n(-\theta)| = 0.
\end{gather}
\end{Corollary}

In view of \eqref{QC} it follows in particular that the random walk $\mathcal{X}$ possesses the SRLP if $M_\infty(\eta)=\infty$, which readily leads to some further sufficient conditions. Indeed, choosing $\theta=\eta$ and defining $L(\eta)$ in analogy with \eqref{L} we have
\begin{gather*}
L(\eta) = \sum_{j\geq 0} \frac{1}{p_j\pi_jQ_j(\eta)Q_{j+1}(\eta)},
\end{gather*}
so, in analogy with \eqref{recurrent}, Corollary~\ref{corollary} leads to
\begin{gather}\label{Leta}
\mathcal{X} \ \mbox{is~aperiodic~and} \ L(\eta) = \infty \ \Rightarrow \ \lim_{n\to\infty}|Q_n(\eta)/Q_n(-\eta)| = 0.
\end{gather}
By \eqref{mon1} we have $L(\eta)\geq L(1)\equiv L$ so the premise in \eqref{Leta} certainly prevails if $\mathcal{X}$ is aperiodic and recurrent. When $L(\eta)=\infty$ the random walk $\mathcal{X}$ is called $\eta$-{\em recurrent} (see \cite{DS95a} for more information). The conclusion that $\eta$-recurrence is sufficient for an aperiodic random walk to possess the SRLP is not surprising, since Pruitt \cite[Theorem~2]{P65} already established this result in the more general setting of symmetrizable Markov chains.

Another sufficient condition for the conclusion in \eqref{rho0eq} is obtained in analogy with~\eqref{rjpj}, namely
\begin{gather*}
\sum_{j\geq 0} \frac{r_jQ_j(\eta)}{p_jQ_{j+1}(\eta)} = \infty \ \Rightarrow \ \lim_{n\to\infty}|Q_n(\eta)/Q_n(-\eta)| = 0.
\end{gather*}
Since, by \eqref{mon1}, $Q_{j+1}(\eta)\leq Q_j(\eta)$ it follows in particular that
\begin{gather}\label{sumrj}
\sum_{j\geq 0} \frac{r_j}{p_j} = \infty \ \Rightarrow \ \lim_{n\to\infty}|Q_n(\eta)/Q_n(-\eta)| = 0.
\end{gather}
Interestingly, we have thus verified a passing remark by Karlin and McGregor \cite[p.~77]{KM59} to the effect that the premise in~\eqref{sumrj} is sufficient for the SRLP.

We now turn to the third lemma needed for the proof of Proposition~\ref{proposition}, which concerns the behaviour of $M_n(\theta)$ as a function of $\theta$.

\begin{Lemma}\label{M} Let $\eta\leq\theta_1\leq\theta_2$, then, for all $n$, $M_n(\theta_1)\geq M_n(\theta_2)$.
\end{Lemma}

\begin{proof}First consider an arbitrary random walk with parameters $p_j$, $q_j$, and $r_j$, $j\in\mathcal{N}$. Let $n$ be fixed and write
\begin{gather*}
M_n = \sum_{j=0}^n\frac{1}{p_j\pi_j}\sum_{k=0}^j r_k\pi_k, \qquad n=0,1,\dots.
\end{gather*}
Suppose that in the single state $\ell$, $0\leq\ell\leq n$, the transition probabilities $p_\ell$, $q_\ell$ and $r_\ell$ are changed into the one-step (random walk) transition probabilities $p'_\ell$, $q'_\ell$ and $r'_\ell$ satisfying, besides the usual requirements,
\begin{gather}\label{change}
p'_\ell\leq p_\ell,\qquad q'_\ell\geq q_\ell \qquad \mbox{and} \qquad r'_\ell\geq r_\ell.
\end{gather}
Let $M_n'$ denote the value of $M_n$ after the change. A somewhat tedious but straightforward calculation then yields that
\begin{gather*}
M'_n = M_n + \left\{(c_1-1)\sum_{k=0}^{\ell-1}r_k\pi_k + (c_1c_2-1)r_\ell\pi_\ell\right\}\sum_{j=\ell}^n\frac{1}{p_j\pi_j},
\end{gather*}
where $c_1$ and $c_2$ are constants satisfying
\begin{gather*}
q_\ell c_1= \frac{p_\ell q'_\ell}{p'_\ell} \qquad \mbox{and} \qquad r_\ell c_2= \frac{p_\ell r'_\ell}{p'_\ell}.
\end{gather*}
The values of $c_1$ when $\ell=0$ and $c_2$ when $r_\ell=0$ are clearly irrelevant, but let us choose $c_1=1$ and $c_2=1$ in these cases. Then, under the given circumstances, we always have $c_1\geq 1$ and $c_2\geq 1$, and hence $M'_n\geq M_n$.

Back to the setting of the lemma we note that if $\eta\leq\theta_1 < \theta_2$, then $r_j(\theta_1) \geq r_j(\theta_2)$, and, by~\eqref{etaleq},
\begin{gather*}
q_j(\theta_1) = \frac{Q_{j-1}(\theta_1)}{Q_j(\theta_1)}\frac{q_j}{\theta_1} > \frac{Q_{j-1}(\theta_2)}{Q_j(\theta_2)}\frac{q_j}{\theta_1} > \frac{Q_{j-1}(\theta_2)}{Q_j(\theta_2)}\frac{q_j}{\theta_2}= q_j(\theta_2),
\qquad j>0.
\end{gather*}
Since $p_j(\theta)+q_j(\theta)+r_j(\theta)=1$, it follows that
\begin{gather}\label{p12}
p_j(\theta_1)<p_j(\theta_2).
\end{gather}
Now let $p_j= p_j(\theta_2)$, $q_j = q_j(\theta_2)$, $r_j= r_j(\theta_2)$ for all $j\in \mathcal{N}$ and suppose we perform the change operation with $p'_\ell=p_\ell(\theta_1)$, $q'_\ell=q_\ell(\theta_1)$ and $r'_\ell=r_\ell(\theta_1)$ (so that \eqref{change} is satisfied) successively for $\ell = 0,1,\dots,n$. Letting $M^{(\ell)}$ be the value into which $M^{(0)}:=M_n(\theta_2)$ has been transformed after the $\ell$th change operation, we then obviously have
\begin{gather*}
M_n(\theta_1) = M^{(n)} \geq M^{(n-1)}\geq \cdots \geq M^{(1)}\geq M^{(0)}=M_n(\theta_2),
\end{gather*}
which was to be proven.
\end{proof}

We have now gathered sufficient information to draw our conclusions in the final section, after noting as an aside that \eqref{p12} leads to a strengthening of~\eqref{etaleq}, namely
\begin{gather*}
\eta \leq x < y \ \Rightarrow \ xQ_n(x)Q_{n+1}(y) > yQ_n(y)Q_{n+1}(x).
\end{gather*}

\section{Proof of Theorem~\ref{theorem} and concluding remarks}\label{conc}

Choosing $\theta_1 = \eta$ and $\theta_2=1$ in Lemma~\ref{M} we conclude that $M_n(\eta)\geq M_n(1)$ for all $n$. Hence $M_\infty(\eta)\geq M_\infty(1)$, so that
\begin{gather*}
M_\infty(1) = \sum_{j\geq 0} \frac{1}{p_j\pi_j} \sum_{k=0}^j r_k\pi_k =\infty \ \Rightarrow \ M_\infty(\eta) = \infty,
\end{gather*}
which, by Corollary~\ref{corollary}, leads to Proposition~\ref{proposition}.

It seems unlikely that there are values of $\theta_1$ and $\theta_2$ such that $\eta<\theta_1<\theta_2$ and $M_\infty(\theta_1)=\infty$, but $M_\infty(\theta_2)<\infty$, since there do not seem to be values of $x> \eta$ that are ``special'' in any sense. So we conjecture that $M_\infty(\theta_1)$ and $M_\infty(\theta_2)$ converge or diverge together. It is tempting to go one step further by extending this conjecture to $\eta\leq\theta_1<\theta_2$. Maintaining the conjecture in~\cite{DS95b} that also the reverse implication in \eqref{QC} is valid, we would then arrive at the conjecture that~\eqref{condition} is not only sufficient but also necessary for~$\mathcal{X}$ to possess the SRLP. However, this not correct, since it is possible to have $M_\infty(\eta)=\infty$ and $M(1)<\infty$ simultaneously, as the next example shows.

\begin{Example}Consider a random walk $\mathcal{\tilde{X}}$ determined by one-step transition probabilities~$\tilde{p}_j$,~$\tilde{q}_j$ and $\tilde{r}_j$ with $\tilde{r}_0>0$ and $\tilde{r}_j=0$ for $j>0$. Quantities associated with $\mathcal{\tilde{X}}$ will be indicated by a~tilde. We will assume that $\mathcal{\tilde{X}}$ is {\em recurrent}, so that $\tilde{\eta}=1$. Now let $\alpha>1$ and define
\begin{gather}\label{probaux2}
p_j := \frac{\tQ_{j+1}(\alpha)}{\tQ_j(\alpha)}\frac{\tilde{p}_j}{\alpha},\qquad
r_j := \frac{\tilde{r}_j}{\alpha},\qquad q_{j+1} := \frac{\tQ_j(\alpha)}{\tQ_{j+1}(\alpha)}\frac{\tilde{q}_{j+1}}{\alpha},
\qquad j \in \mathcal{N}.
\end{gather}
These quantities, like those in~\eqref{probaux}, can be interpreted as the one-step transition probabilities of a new random walk $\mathcal{X}$, say. In what follows we associate quantities without tilde with $\mathcal{X}$. In analogy with~\eqref{newQ} and~\eqref{newpi} we thus have $Q_n(x)=\tQ_n(\alpha x)/\tQ(\alpha)$ and $\pi_n=\tilde{\pi}_n\tQ_n^2(\alpha)$. Also, $\eta = \tilde{\eta}\alpha^{-1} = \alpha^{-1}<1$, so that $\mathcal{X}$ must be transient. Next, letting $M_n(\theta)$ be defined as in \eqref{Mntheta} and \eqref{probaux} where $p_j$, $q_j$ and $r_j$ are given by \eqref{probaux2}, we have
\begin{gather*}
M_\infty(1) = r_0 \sum_{j\geq 0} \frac{1}{p_j\pi_j} < \infty,
\end{gather*}
since $\mathcal{X}$ is transient. But on the other hand
\begin{gather*}
M_\infty(\eta) = M_\infty\big(\alpha^{-1}\big) = r_0\sum_{j\geq 0} \frac{1}{p_j\pi_jQ_j\big(\alpha^{-1}\big)Q_{j+1}\big(\alpha^{-1}\big)}=\tilde{r}_0\sum_{j\geq 0} \frac{1}{\tilde{p}_j\tilde{\pi}_j} = \infty,
\end{gather*}
since $\mathcal{\tilde{X}}$ is recurrent.
\end{Example}

We have already encountered several known sufficient conditions for the random walk $\mathcal{X}$ to possess the SRLP. In particular, $\eta$-recurrence~-- and thus recurrence, which is simply $1$-re\-cur\-rence~-- was shown to be sufficient in~\eqref{Leta}. Also, in view of~\eqref{ineq} we regain directly from Theorem~1 Karlin and McGregor's claim on \cite[p.~77]{KM59}
\begin{gather*}
\sum_{j\geq 0} \frac{r_j}{p_j} = \infty \ \Rightarrow \ \mathcal{X}
~\mbox{possesses~the~SRLP},
\end{gather*}
referred to after \eqref{sumrj}. Several authors (see \cite[p.~77]{KM59}, \cite[Corollary~3.2]{DS95b}) have shown that for the SRLP to prevail it is sufficient that $r_j>\delta> 0$ for $j$ sufficiently large, but this condition is evidently weaker than the previous one.

\pdfbookmark[1]{References}{ref}
\LastPageEnding

\end{document}